\documentclass{article}
\usepackage[letterpaper, margin=1.25in]{geometry}
\usepackage{amsmath,amsthm,amscd,color}
\usepackage{amssymb}
\usepackage{amsfonts}
\usepackage{latexsym}
\usepackage{mathrsfs}
\usepackage{hyperref}
\usepackage{geometry}
\usepackage{fancyvrb}
\usepackage{dsfont}
\usepackage[dvipsnames]{xcolor}
\usepackage{tikz-cd}
\usepackage{rotating}
\usepackage{graphicx,diagbox}
\usepackage{caption}
\usepackage{subcaption}

\newtheorem{theorem}[equation]{Theorem}
\newtheorem{lemma}[equation]{Lemma}
\newtheorem{proposition}[equation]{Proposition}
\newtheorem{corollary}[equation]{Corollary}
\newtheorem{conjecture}[equation]{Conjecture}

\theoremstyle{remark}
\newtheorem{definition}[equation]{Definition}

\newtheorem*{remark}{\bf Remark}

\numberwithin{equation}{section}
\numberwithin{table}{section}

\newcommand{\wt}{\widetilde}
\newcommand{\wh}{\widehat}

\renewcommand{\frak}{\mathfrak}

\DeclareMathOperator{\HF}{HF}

\DeclareMathOperator{\im}{Im}

\DeclareMathOperator{\coker}{coker}

\DeclareMathOperator{\Span}{Span}
\DeclareMathOperator{\rank}{rank}

\DeclareMathOperator{\Spin}{Spin}
\DeclareMathOperator{\nd}{nd}
\DeclareMathOperator{\PD}{PD}

\newcommand{\Z}{{\mathbb{Z}}}

\newcommand{\C}{{\mathbb{C}}}
\newcommand{\A}{{\mathbb{A}}}
\newcommand{\B}{{\mathbb{B}}}

\newcommand{\F}{{\mathbb{F}}}

\newcommand{\X}{{\mathbb{X}}}

\title{Heegaard Floer Surgery Formula and Cosmetic Surgeries}
\author{ALAN DU\thanks{1200 E California Blvd, Pasadena CA 91125 USA, aydu@caltech.edu} \\
Caltech}
\date{February 14, 2026}

\begin{document}
\maketitle
\begin{abstract}
Two Dehn surgeries on a knot are called cosmetic if they yield homeomorphic three-manifolds. We show for a certain family of null-homologous knots in any closed orientable three-manifold, if the knot admits cosmetic surgeries with a pair of positive surgery coefficients, then the coefficients are both greater than $1$. In addition, for this family of knots, we show that $+1/q$ Dehn surgery is not homeomorphic to the original three-manifold. The proofs of these results use the mapping cone formula for the Heegaard Floer homology of Dehn surgery in terms of the knot Floer homology of the knot; we provide a new proof of this formula for integer surgeries in $\Spin^c$ structures with nontorsion first Chern class.
\end{abstract}


\section{Introduction}

Heegaard Floer homology is a powerful invariant of closed oriented three-manifolds introduced by Ozsv\'ath and Szab\'o \cite{oz-sz-hf}. It has been used to prove many results about three-manifold topology, especially Dehn surgery. In fact, the Heegaard Floer homology of any surgery on a knot can be computed from the knot Floer complex of the knot \cite{oz-sz-integer,oz-sz-rational,zemke-hf-surgery-formula}. 

In this paper, we will use Heegaard Floer homology to study cosmetic surgeries on a certain class of knots in arbitrary closed oriented three-manifolds. We recall the definition of cosmetic surgeries.

\begin{definition}
If two Dehn surgeries on a knot yield homeomorphic manifolds, then these two surgeries are \textit{cosmetic}. If in addition, the two surgeries are orientation-preserving homeomorphic, then these two surgeries are \textit{purely cosmetic}.
\end{definition}

Cosmetic surgeries are very rare; in fact, the following cosmetic surgery conjecture has been proposed:

\begin{conjecture}[\cite{kirby-problem-list-1995} Problem 1.81]
Suppose $K$ is a knot in a closed oriented three-manifold $Y$ such that $Y-K$ is irreducible and not homeomorphic to the solid torus. If two different Dehn surgeries on $K$ are purely cosmetic, then there is a homeomorphism of $Y-K$ which takes one slope to the other.    
\end{conjecture}

The existence of cosmetic surgeries has been ruled out in many cases for knots in $S^3$. For instance, Ozsv\'ath and Szab\'o proved that if $K\subset S^3$ admits cosmetic rational surgeries, then either the resulting manifold is an $L$-space or the surgery slopes have opposite signs \cite{oz-sz-rational}. Much more is known when considering purely cosmetic surgeries. Hanselman proved that any pair of purely cosmetic Dehn surgeries on knots in $S^3$ must have slopes that are $\pm 2$ or $\pm 1/q$ for some integer $q$ \cite{hanselman-hf-cosmetic}.

Less is known for knots in arbitrary three-manifolds. Ni found an obstruction to cosmetic surgeries on knots in three-manifolds satisfying a certain condition on their Thurston norms \cite{ni-thurston-norm-cosmetic}. 

The main theorem in this paper is an analogue of Ozsv\'ath and Szab\'o's result. We consider a class of knots which were shown by Ni to have special properties on their knot Floer homology.

\begin{definition}
Let $Y$ be a closed, oriented three-manifold, and $K\subset Y$ a null-homologous knot. We say that $K$ is \textit{blow-up bounding} if there exists a taut Seifert surface $F$ for $K$, some natural number $N$, and a properly-embedded smooth disk $G$ inside $X=(Y\times[0,1])\# N\overline{\C P^2}$ with $\partial G=K\times\{0\}$ and $[G]=\iota_*[F]\in H_2(X,K\times\{0\})$ for $\iota:(Y,K)\to(X,K\times\{0\})$ the inclusion.

We say that $K$ is \textit{blow-up unknotted} if $K$ is obtained from the unknot in $Y$ by a sequence of crossing changes that change a positive crossing to a negative one. 
\end{definition}

By Kirby calculus, $K$ is blow-up bounding if $K$ is blow-up unknotted.

\begin{remark}
The property of $K$ being blow-up bounding implies that $K$ is null homotopic.    
\end{remark}

\begin{theorem}
Let $K\subset Y$ be a blow-up unknotted knot, and suppose $0<r<s$ are distinct positive rational numbers with the property that $Y_r(K)\cong\pm Y_s(K)$. Then $r,s>1$ or $K$ is the unknot. 
\end{theorem}

We also provide progress on the knot complement problem for blow-up unknotted knots.

\begin{theorem}
Let $K\subset Y$ be a nontrivial blow-up bounding knot and $q\ge 2$, then $Y_{1/q}(K)\ncong\pm Y$. Furthermore, if $K$ is a nontrivial blow-up unknotted knot, then $Y_1(K)\ncong\pm Y$. 
\end{theorem}

The proofs of these theorems rely on a mapping cone formula that relates the Heegaard Floer homology of Dehn surgery on null-homologous knots in arbitrary closed, oriented three-manifolds with the knot Floer chain complex. This result was proven for Heegaard Floer homology in $\Spin^c$ structures whose first Chern classes are torsion by Ozsv\'ath and Szab\'o, and the formula in nontorsion $\Spin^c$ structures for zero surgeries was proven by Ni \cite{oz-sz-integer,oz-sz-rational,ni-property-g}. A recent result by Zemke provides a very general surgery formula via an equivalence of certain Lagrangians with local systems in the Fukaya category \cite{zemke-hf-surgery-formula}. Perhaps of independent interest, we provide a conceptually-simpler proof of the integer surgery formula for nontorsion $\Spin^c$ structures, which generalizes Ozsv\'ath and Szab\'o's and Ni's results.

\section{Preliminaries on Heegaard Floer Homology}

In this section, we establish some notation and collect some results, most of which are copied from \cite{oz-sz-integer}, that will be needed in later sections.

Let $W:Y_1\to Y_2$ be an oriented cobordism, with $Y_1,Y_2$ connected. Given $\frak s\in\Spin^c(W)$, there is a homomorphism
$$F_{W,\frak s}^\circ: HF^\circ(Y_1,\frak s\vert Y_1)\to HF^\circ(Y_2,\frak s\vert Y_2),$$ induced by a chain map
$$f_{W,\frak s}^\circ: CF^\circ(Y_1,\frak s\vert Y_1)\to CF^\circ(Y_2,\frak s\vert Y_2).$$

Among other properties, these maps satisfy a conjugation invariance. For $\frak t$ a $\Spin^c$ structure (in either three or four dimensions), let $\overline{\frak t}$ denote its conjugate. There is an isomorphism $\frak I_Y:HF^\circ(Y,\frak t)\to HF^\circ(Y,\overline{\frak t})$.

\begin{theorem}[\cite{oz-sz-4-mflds} Theorem 3.6]\label{thm:conjugation-invariance}
Let $W$ be a cobordism from $Y_1$ to $Y_2$ as above. Then
$$F_{W,\frak s}^\circ=\frak I_{Y_2}\circ F_{W,\overline{\frak s}}^\circ \circ\frak I_{Y_1}.$$
\end{theorem}

Given a chain map $f:A\to B$, let $MC(f)$ be the mapping cone of $f$, namely, $MC(f)=A\oplus B$ with the differential
$$\partial=\begin{pmatrix}
    \partial_A & 0 \\
    f & \partial_B
\end{pmatrix}.$$

\begin{theorem}[\cite{oz-sz-integer} Theorem 3.1]\label{thm:oz-sz-les}
Let $Y$ be a closed, oriented three-manifold, $K\subset Y$ a null-homologous knot. Fix an integer $n$ and a positive integer $m$. Then, there is a long exact sequence
$$\cdots\to HF^+(Y_n(K))\to HF^+(Y_{m+n}(K))\to\bigoplus^m HF^+(Y)\to\cdots$$ and also a corresponding exact sequence using $\widehat{HF}$ in place of $HF^+$. There are $\Z[U]$-equivariant chain maps
$$f_1^+:CF^+(Y_n(K))\to CF^+(Y_{m+n}(K))$$
$$f_2^+:CF^+(Y_{m+n}(K))\to\bigoplus^m CF^+(Y)$$
$$f_3^+:\bigoplus^m CF^+(Y)\to CF^+(Y_n(K))$$
inducing the maps in the long exact sequence, and $\Z[U]$-equivariant quasi isomorphisms
$$\phi^+: CF^+(Y_n(K))\to MC(f_2^+)$$
$$\psi^+: MC(f_2^+)\to CF^+(Y_n(K)).$$
\end{theorem}

Knot Floer homology associates to a knot $K$ in a closed, oriented three-manifold $Y$ a $\Z\oplus\Z$-filtered $\Z[U]$-complex $C=CFK^\infty(Y,K)$ with differential $\partial$. We will use the knot Floer chain complex to construct chain complexes $\X^+(n)$ and $\X^+_{p/q}$ that we will show in sections 3 and 4 are quasi-isomorphic to the plus versions of the Heegaard Floer chain complexes for integral $n$-surgery or rational $p/q$-surgery on the knot $K$ in $Y$, respectively. We will also make use of the $CF^\delta$ variant of Heegaard Floer homology defined in \cite{oz-sz-integer}. For any nonnegative integer $\delta$, denote by $CF^\delta(Y)$ the subcomplex of $CF^+(Y)$ which is the kernel of multiplication by $U^{\delta+1}$. Similarly, there are $\delta$ versions of the other constructions in this section defined analogously to the plus versions, and the results for the plus version also hold for the $\delta$ version. The hat version, denoted $\wh{CFK}(Y,K)$ and similarly, is the $\delta$ version for $\delta=0$. 

Write the two factors of the $\Z\oplus\Z$ filtration as $(i,j)$. There is a $U$-action on $C$ that is an isomorphism of bifiltration $(-1,-1)$.

We will often use the following fact, for example see \cite[Reduction Lemma]{hedden-watson-geo-botony}.

\begin{lemma}\label{lem:reduction}
With coefficients in any field $\F$, the knot Floer complex $(C,\partial)$ for any knot $K\subset Y$ is $(\Z\oplus\Z)$-filtered chain homotopy equivalent to a chain complex $(\overline C,\overline\partial)$ such that the restriction of the differential $\overline\partial$ to any $\overline C\{i,j\}$ is $0$ and $\overline C\{0,j\}$ is isomorphic to $\wh{HFK}(K,j)$.
\end{lemma}

Knot Floer homology is known to detect the genus of the knot, and this is true for knots in general three-manifolds.

\begin{theorem}\cite[Theorem 2.4]{ni-thurston-norm-cosmetic}\label{thm:hfk-detects-genus}
For any field $\F$, if $K\subset Y$ is a knot of genus $g$, then $$\wh{HFK}(Y,K,s;\F)=0$$ if $|s|>g$. Furthermore, $\wh{HFK}(Y,K,g;\F)\ne 0$.
\end{theorem}

Define $A_s^+=C\{\max(i,j-s)\ge 0\}$ and $B^+=C\{i\ge 0\}$; these are chain complexes because the differential decreases both filtrations. There are two canonical chain maps $v_s^+:A_s^+\to B^+$ and $h_s^+:A_s^+\to B^+$, where $v_s^+$ is the vertical projection, and $h_s^+$ first projects $A_s^+$ to $C\{j\ge s\}$, then maps to $C\{j\ge 0\}$ via $U^s$, and finally maps to $B^+$ via a chain homotopy equivalence $C\{j\ge 0\}\to C\{i\ge 0\}$ \cite{oz-sz-integer}.

For all $s\in\Z$, there is a chain homotopy equivalence $\phi$ between $A_s^+$ and $A_{-s}^+$: first, the $U^s$ map is an isomorphism
$$A_s^+=C\{\max(i,j-s)\ge0\}\to C\{\max(i,j-s)\ge-s\}, $$ then the change of basepoint map that swaps the roles of $w$ and $z$ in the doubly-pointed Heegaard diagram is a chain homotopy equivalence
$$C\{\max(i,j-s)\ge-s\}\to C\{\max(i,j+s)\ge0\}=A_{-s}^+.$$ Let $\phi$ be the composition of these two maps. It is clear from the definitions that the following lemma holds: 

\begin{lemma}\label{lem:v-h-equiv}
For all $s\in\Z$, $v_s^+$ and $h_s^+$ are chain homotopy equivalent to $h_{-s}^+\circ\phi$ and $v_{-s}^+\circ\phi$, respectively.
\end{lemma}

Since all maps involved are $U$-equivariant, the above lemma also holds for the $CF^\delta$ version of knot Floer homology for any $\delta\ge 0$.

\begin{lemma}\label{lem:genus-finiteness}
Let $g$ be the genus of $K$, then for all $\delta\ge 0$ and $s\ge g+\delta$, $v_s^\delta$, and $h_{-s}^\delta$ induce isomorphisms on homology, and $v_{-s-1}^\delta$ and $h_{s+1}^\delta$ are nullhomotopic.
\end{lemma}
\begin{proof}
It is shown in the proof of the adjunction inequality in \cite[Theorem 5.1]{oz-sz-hfk} that one can construct a special Heegaard diagram for $(Y,K)$ such that all generators of $\wh{CFK}(Y,K)$ lie in Alexander grading greater than or equal to $-g$. Thus, all generators for $A_s^\delta$ with $i\ge 0$ must have $j\ge -g$, from which it follows that $h_{-s}^\delta$ induces an isomorphism on homology and $v_{-s-1}^\delta$ is null-homotopic. By the $CF^\delta$ version of Lemma~\ref{lem:v-h-equiv}, we obtain the statements about $v_s^\delta$ and $h_{s+1}^\delta$.
\end{proof}

Let $K\subset Y$ be null-homologous. Let $F$ be a Seifert surface for $K$ and let $\hat F$ be the closed surface obtained by capping off $\partial F$ with a disk. Fix a $\Spin^c$ structure $\frak t$ over $Y$, which need not have torsion first Chern class, and let $n\in\Z$. The space of $\Spin^c$ structures over $Y_n(K)$ which are $\Spin^c$-cobordant to $\frak t$ over the natural two-handle cobordism $W_n'(K):Y_n(K)\to Y$ is identified with $\Z/n\Z$ by $\frak u\mapsto i$, where $\frak s$ is an extension of $\frak u$ over $W_n'(K)$ with the property that
$$\langle c_1(\frak s),[\hat F]\rangle -n\equiv 2i \pmod{2n}.$$
For $i\in\Z/n\Z$, let $CF^+(Y_n(K),i,\frak t)$ denote the corresponding summand of $CF^+(Y_n(K))$. Let $A_{s,\frak t}^+,B_{\frak t}^+$ denote the corresponding chain complexes associated to the knot filtration on $CF^+(Y,\frak t)$ \cite{oz-sz-integer}.

Let $$\A_{i,\frak t}^+=\bigoplus_{s\in\Z\vert s\equiv i\pmod n} A_{s,\frak t}^+ \hspace{5mm}\text{and}\hspace{5mm} \B_{\frak t}^+=\bigoplus_{s\in\Z\vert s\equiv i\pmod n} B_{\frak t}^+$$ and let $D_{n,i,\frak t}^+:\A_{i,\frak t}^+\to\B_{i,\frak t}^+$ be the map
$$D_{n,i,\frak t}^+(\{a_s\})=\{b_s\},$$ where here
$$b_s=h_{s-n}^+(a_{s-n})+v_s^+(a_s).$$ 
Let $\X_{i,\frak t}^+(n)$ denote the mapping cone of $D_{n,i,\frak t}^+$; we prove in the next section that this chain complex computes $CF^+(Y_n(K),i,\frak t)$ \cite{oz-sz-integer}.

Taking the direct sum of the above chain complexes over all $\frak t\in\Spin^c(Y)$ and $i\in\Z/n\Z$ gives $\A^+$, $\B^+$, and the map $D_n^+:\A^+\to\B^+$. Let $\X^+(n)$ denote the mapping cone of $D_n^+$.

Let $K\subset Y$ be a null-homologous knot and let $F$ be a Seifert surface for $K$ in the homology class $\phi$. Given $n>0$ and $i\in\Z/n\Z$, let $\wh F_n\subset W_n'(K)$ be the closed surface obtained by capping off $\partial F$ with a disk. Let $\wh \phi_n=[\wh F_n]\in H_2(W_n'(K))$. Let $k$ be an integer satisfying $k\equiv i\pmod n$ and $|k|\le n/2$. Let $\frak x_k,\frak y_k\in\Spin^c(W_n'(K))$ satisfy $\frak x_k|Y=\frak y_k|Y=\frak t$ and
$$\langle c_1(\frak x_k),\hat{\phi}_n\rangle = 2k-n, \hspace*{5mm} \langle c_1(\frak y_k),\hat \phi_n\rangle = 2k+n,$$ and let $\frak t_i=\frak x_k|Y_n(K)=\frak y_k|Y_n(K)$. 

\begin{theorem}[\cite{oz-sz-integer} Theorem 2.3]\label{thm:oz-sz-2.3}
When $n\ge 2g(F)$, the chain complex $CF^+(Y_n(K),\frak t_i)$ is represented by the chain complex $A_{k,\frak t}^+$, in the sense that there is an isomorphism 
$$\Psi_n^+:CF^+(Y_n(K),\frak t_i)\to A_{k,\frak t}^+.$$ Moreover, the squares in the diagrams of Figure~\ref{fig:v-h-commutative-diagrams} commute.

\begin{figure}[h]
\centering
\begin{subfigure}[b]{0.45\textwidth}
\centering
\begin{tikzcd}
{CF^+(Y_n(K),\frak t_i)} \arrow[d, "\Psi_n^+"] \arrow[rr, "{f_{W_n'(K),\frak x_k}^+}"]  & {} & {CF^+(Y,\frak t)} \arrow[d, "="] \\
A_{k,\frak t}^+ \arrow[rr, "v_k^+"]                                                               & {} & B_{\frak t}^+                             
\end{tikzcd}
\end{subfigure}
\hfill
\begin{subfigure}[b]{0.45\textwidth}
\centering
\begin{tikzcd}
{CF^+(Y_n(K),\frak t_i)} \arrow[d, "\Psi_n^+"] \arrow[rr, "{f_{W_n'(K),\frak y_k}^+}"]  & {} & {CF^+(Y,\frak t)} \arrow[d, "="] \\
A_{k,\frak t}^+ \arrow[rr, "h_k^+"]                                                               & {} & B_{\frak t}^+.                             
\end{tikzcd}
\end{subfigure}
\caption{Commutative diagrams of the maps $v_k^+$ and $h_k^+$}
\label{fig:v-h-commutative-diagrams}

\end{figure}

\end{theorem}

The $CF^\delta$ version of the theorem holds with the same proof.

In the case of rational surgeries, let $p,q$ be a pair of relatively prime integers. For $\frak t\in\Spin^c(Y)$ and $i\in\Z/p\Z$, denote
$$\A_{p/q,i,\frak t}^+=\bigoplus_{r\in\Z} (r,A_{\lfloor(i+pr)/q\rfloor,\frak t}^+(K)) \hspace{5mm}\text{and}\hspace{5mm} \B_{p/q,i,\frak t}^+=\bigoplus_{r\in\Z} (r,B_{\frak t}^+).$$ Let $D_{p/q,i,\frak t}^+:\A_{p/q,i,\frak t}^+\to\B_{p/q,i,\frak t}^+$ be the map
$$D_{p/q,i,\frak t}^+\{(r,a_r)\}=\{(r,b_r)\},$$ where
$$b_r=v_{\lfloor(i+pr)/q\rfloor}^+(a_r)+h_{\lfloor(i+p(r-1))/q\rfloor}^+(a_{r-1}).$$
Let $\X_{p/q,i,\frak t}^+$ denote the mapping cone of $D_{p/q,i,\frak t}^+$; in section 4, we prove that the homology of this chain complex is isomorphic to $HF^+(Y_{p/q}(K),i,\frak t)$ \cite{oz-sz-rational}. 

As before, direct summing over all $\frak t\in\Spin^c(Y)$ and $i\in\Z/p\Z$ gives $\A_{p/q}^+$, $\B_{p/q}^+$, and $D_{p/q}^+:\A_{p/q}^+\to\B_{p/q}^+$. Let $\X_{p/q}^+$ denote the mapping cone of $D_{p/q}^+$. In the case where $p/q$ is clear from the context, we will drop the subscript $p/q$ from the notation.

\section{Integer Surgeries}

In this section, we provide a proof of the mapping cone formula for integer surgeries in the case where the $\Spin^c$ structure has nontorsion first Chern class. The formula for torsion $\Spin^c$ structures was originally proven by Ozsv\'ath and Szab\'o \cite{oz-sz-rational}. Our proof combines ingredients from their argument as well as Ni's proof of the zero surgery formula.

Let $K\subset Y$ be a null-homologous knot. Fix a $\Spin^c$ structure $\frak t$ over $Y$ with nontorsion first Chern class. 

Fix $n>0$. The long exact sequence in Theorem~\ref{thm:oz-sz-les} decomposes over $\Spin^c$ structures on $Y$. Setting $m=nk$ for large $k>0$, it also decomposes over $\Spin^c$ structures on $Y_n(K)$, i.e. given $i\in\Z/n\Z$, $CF^+(Y_n(K),i,\frak t)$ is quasi-isomorphic to $MC(f_{2,i,\frak t}^+)$ where
\begin{align}\label{eq:f2}
f_{2,i,\frak t}^+=\sum_{\frak s_s\in\mathcal X_i}f_{W_{n(k+1)}'(K),\frak s_s}^+:CF^+(Y_{n(k+1)}(K), [i], \frak t)\to CF^+(Y,\frak t),
\end{align}
\begin{align*}
\mathcal X_i=\{\frak s_s\in\Spin^c(W'_{n(k+1)}(K))\ &\vert\ \frak s_s\vert_Y=\frak t, \\
&\langle c_1(\frak s_s),[\widehat{F}]\rangle-n(k+1)\equiv 2s\pmod{2n(k+1)}, \\
&\text{and }s\equiv i\pmod n\},
\end{align*}
$$CF^+(Y_{n(k+1)}(K), [i], \frak t)=\bigoplus_{s\in\Z/n(k+1)\Z\ \vert\ s\equiv i\pmod n}CF^+(Y_{n(k+1)}(K), s, \frak t).$$

A strengthening of Proposition 3.5 in \cite{ni-property-g} shows the following:

\begin{proposition}[\cite{ni-property-g} Proposition 3.5]\label{thm:ni-3.5}
Let $n>0,$ $-n/2\le k< n/2$, and $p=\max\{k+g(F),-2k+g(F)\}$. Suppose that $n\ge\max\{p,2g(F)\}$, $\frak s\in\mathcal X_k\setminus\{\frak x_k,\frak y_k\},$ then the map $F_{W_n'(K),\frak s}^-$ factorizes through the map
$$U^{n-p}:\HF^-(Y_n(K),k,\frak t)\circlearrowright.$$
Moreover, if $c_1(\frak t)$ is nontorsion, then $F_{W_n'(K),\frak s}^+$ factorizes through
$$U^{n-p}:\HF^+(Y_n(K),k,\frak t)\circlearrowright.$$
\end{proposition}

This proposition is exactly the same as \cite{ni-property-g} Proposition 3.5 when $k\le 0$, but when $k>0$, we improve the bound on the power of $U$ from $n-2k-g(F)$ to $n-p=n-k-g(F)$. The proof of the improved bound is just checking that when $k>0$, the inequality
$$\frac{l(l+1)n}{2}-lk-g\ge n-2|k|-g$$ in Ni's proof can be improved to
$$\frac{l(l+1)n}{2}-lk-g\ge n-k-g.$$

By the same proof, the proposition also holds for the $\delta$ version of Heegaard Floer homology for any $\delta\ge 0$.

\begin{theorem}\label{thm:mc-integer}
For each $i\in\Z/n\Z$, the mapping cone $\X_{i,\frak t}^+(n)$ of 
$$D_{n,i,\frak t}^+:\A_{i,\frak t}^+\to\B_{i,\frak t}^+$$ is quasi-isomorphic to $CF^+(Y_n(K),i,\frak t)$ as $\Z[U]$-modules.
\end{theorem}

\begin{proof}[Proof when $n>0$]
We first prove the theorem for the $CF^\delta$ version of knot Floer homology for any $\delta\ge0$, that is, the mapping cone $\X_{i,\frak t}^\delta(n)$ is quasi-isomorphic to $CF^\delta(Y_n(K),i,\frak t)$ as $\Z[U]$-modules. 

Recall that $CF^\delta(Y_n(K),i,\frak t)$ is quasi-isomorphic to $MC(f_{2,i,\frak t}^\delta)$ as in Equation~\ref{eq:f2}. Since $c_1(\frak t)$ is nontorsion, $H_*(A_i^+)$ is a finitely generated abelian group. So $U^q\vert H_*(A_i^+)=0$ when $q$ is greater than a constant $C_1$ independent of $n,k$. Hence, $U^q\vert H_*(A_i^\delta)=0$ for $q>C_1$ for the same constant $C_1$ independent of $n,k,$ and $\delta$. By the $CF^\delta$ version of Theorem~\ref{thm:oz-sz-2.3}, this implies that $U^q\vert HF^\delta(Y_{n(k+1)}(K),i,\frak t)=0$ when $q>C_1$ and $n(k+1)>2g$. Choose $k$ such that $n(k+1)>2(g+\delta)$.

Consider a $\Spin^c$ structure $\frak s_s\in\mathcal X_i$ for some $s\in\Z/n(k+1)\Z$. Recall that we choose representatives
$$-\frac{n(k+1)}{2}\le s<\frac{n(k+1)}{2}$$ for classes in $\Z/n(k+1)\Z$. Suppose $s\ge 0$. Then for $l=\max\{s+g,-2s+g\}$, we have 
$$l\le \frac{n(k+1)}{2}+g\le n(k+1).$$ Proposition~\ref{thm:ni-3.5} implies that $F_{W_{n(k+1)}'(K),\frak s_s}^\delta=0$ when $\frak s_s\in\mathcal X_i\setminus\{\frak x_s,\frak y_s\}$. 

Now for $s<0$, let $\frak s_s\in\mathcal X_i\setminus\{\frak x_s,\frak y_s\}$, then by conjugation invariance, there is a chain homotopy equivalence
$$F_{W_{n(k+1)}'(K),\frak s_s}^\delta\cong F_{W_{n(k+1)}'(K),\overline{\frak s_s}}^\delta.$$ If $\frak s_s=\frak x_s+r\PD[\widehat F]$, then 
$$\overline{\frak s_s}=\overline{\frak x}_{-s}-(r+1)\PD[\widehat F],$$ in particular, $\overline{\frak s_s}$ is not one of the two special $\Spin^c$ structures and $-s>0$. By the above argument, $F_{W_{n(k+1)}'(K),\overline{\frak s_s}}^\delta=0$, hence also $F_{W_{n(k+1)}'(K),\frak s_s}^\delta=0$.

Therefore, $MC(f_{2,i,\frak t}^\delta)$ is quasi-isomorphic to the mapping cone of$$\sum_{s\in\Z/n(k+1)\Z\ :\ s\equiv i \pmod n}\left(f_{W_{n(k+1)}'(K),\frak x_s}^\delta+f_{W_{n(k+1)}'(K),\frak y_s}^\delta\right) : CF^\delta(Y_{n(k+1)}(K),[i],\frak t)\to CF^\delta(Y,\frak t).$$ By Theorem~\ref{thm:oz-sz-2.3}, this chain complex is quasi-isomorphic to the mapping cone of
$$\sum_{s\in\Z/n(k+1)\Z\ :\ s\equiv i \pmod n} (v_s^\delta + h_s^\delta): \bigoplus_{s\in\Z/n(k+1)\Z\ :\ s\equiv i \pmod n} A_s^\delta\to \bigoplus_{s\in\Z/n(k+1)\Z\ :\ s\equiv i \pmod n} B^\delta.$$ We rewrite this as the truncated mapping cone
$$\sum_{-n(k+1)/2\le s<n(k+1)/2\ :\ s\equiv i \pmod n} (v_s^\delta + h_s^\delta).$$ We chose $k$ large enough such that $\frac{n(k+1)}{2}>g+\delta$, so by Lemma~\ref{lem:genus-finiteness}, the above mapping cone is quasi-isomorphic to the mapping cone of
$$\sum_{s\ :\ s\equiv i \pmod n} (v_s^\delta + h_s^\delta): \bigoplus_{s\ :\ s\equiv i \pmod n} A_s^\delta\to \bigoplus_{s\ :\ s\equiv i \pmod n} B^\delta,$$ which by definition is $\X_{i,\frak t}^\delta(n)$.

Therefore, we have the quasi-isomorphism
$$CF^\delta(Y_n(K),i,\frak t)\cong \X_{i,\frak t}^\delta(n)$$ for all $\delta\ge 0$. Applying \cite[Lemma 2.7]{oz-sz-integer}, we conclude that
$$CF^+(Y_n(K),i,\frak t)\cong\X_{i,\frak t}^+(n).$$

\end{proof}

\begin{proof}[Proof of Theorem~\ref{thm:mc-integer} for negative surgery coefficients]

Continue to assume $n>0$, only we consider $Y_{-n}(K)$. Again, we first work with the $\delta$ version for $\delta\ge 0$. In this case, let $W'_{n(k-1)}(K)$ be the cobordism from $Y_{n(k-1)}(K)$ to $Y$. For each $\frak t\in\Spin^c(Y)$, $i\in\Z/n\Z$,
$$HF^\delta(Y_{-n}(K),i,\frak t)\cong H_*\left(MC\left(f_{2,i,\frak t}^\delta:CF^\delta(Y_{n(k-1)}(K),[i],\frak t)\to\bigoplus^{nk}CF^\delta(Y,\frak t)\right)\right),$$ where
$$
f_{2,i,\frak t}^\delta=\sum_{\frak s_s\in\mathcal X_i}f_{W_{n(k-1)}'(K),\frak s_s}^\delta:CF^\delta(Y_{n(k-1)}(K), [i], \frak t)\to CF^\delta(Y,\frak t),
$$
\begin{align*}
\mathcal X_i=\{\frak s_s\in\Spin^c(W'_{n(k-1)}(K))\ &\vert\ \frak s_s\vert_Y=\frak t,  \\
&\langle c_1(\frak s_s),[\widehat{F}]\rangle-n(k-1)\equiv 2s\pmod{2n(k-1)}, \\
&\text{and }s\equiv i\pmod n\},
\end{align*}
$$CF^\delta(Y_{n(k-1)}(K), [i], \frak t)=\bigoplus_{s\in\Z/n(k-1)\Z\ \vert\ s\equiv i\pmod n}CF^\delta(Y_{n(k-1)}(K), s, \frak t).$$

Consider the summand $CF^\delta(Y_{n(k-1)}(K),s,\frak t)$ for 
$$-\frac{n(k-1)}{2}\le s<\frac{n(k-1)}{2},\hspace*{5mm} s\equiv i\pmod n.$$ Then, as above, if $\frak s_s\in\mathcal X_i\setminus\{\frak x_s,\frak y_s\}$, then $F_{W'_{n(k-1)}(K),\frak s_s}^\delta=0$. The conclusion follows again by the same argument as in the positive surgery case.

\end{proof}

\section{Rational Surgeries}

In this section, we sketch the proof of the rational surgery formula:

\begin{theorem}\label{thm:rational-surgery-formula}
Let $K\subset Y$ be a null-homologous knot, and let $p,q$ be a pair of relatively prime integers. Then, for each $\frak t\in\Spin^c(Y)$, $i\in\Z/p\Z$, there is an isomorphism of $\Z[U]$-modules
$$H_*(\mathbb X_{p/q,i,\frak t}^+)\cong HF^+(Y_{p/q}(K),i,\frak t).$$
\end{theorem}

Ozsv\'ath and Szab\'o stated and proved this theorem in the case where $\frak t$ has torsion first Chern class \cite{oz-sz-rational}. In this case, they showed this isomorphism also preserves $\Z$-grading. Their argument still holds when the $\Spin^c$ structure is nontorsion, as long as grading information is ignored.

Let $a=\lfloor p/q\rfloor$ and $p/q=a+r/q$. Consider the Hopf link in $S^3$ and perform $-q/r$ surgery on one component of the link. Let $O_{q/r}\subset -L(q,r)$ be the other component of the link viewed as a knot in the resulting lens space after surgery. By Kirby calculus, $Y_{p/q}(K)$ can be realized by surgery with coefficient $a$ inside the knot $K\#O_{q/r}\subset Y\#(-L(q,r))$ \cite{oz-sz-rational}. Thus, one obtains the rational surgery formula by proving a formula for Morse surgeries on rationally null-homologous knots and then applying a Kunneth principle for connected sums.

Analogously to the case of integer surgeries on null-homologous knots, for $K\subset Y$ a rationally null-homologous knot with an orientation (we use $\overline K$ to denote the oriented knot), $\lambda$ a framing on $K$, and $\frak s\in\Spin^c(Y_\lambda(K))$, define the chain map $D_{\frak s}^+:\A_{\frak s}^+(Y,\overline{K})\to\B_{\frak s}^+(Y,\overline{K})$ as in \cite{oz-sz-rational} Theorem 6.1. Let $\X_{\frak s}^+(Y,\overline{K},\lambda)$ be the mapping cone of $D_{\frak s}^+$ \cite{oz-sz-rational}.

\begin{theorem}[Morse surgery formula \cite{oz-sz-rational} Theorem 6.1]
Let $K\subset Y$ be a rationally null-homologous knot, and let $\lambda$ be a framing on $K$. Let $\frak s\in\Spin^c(Y_\lambda(K))$, then
$$HF^+(Y_\lambda(K),\frak s)\cong H_*(\X_{\frak s}^+(Y,\overline{K},\lambda))$$ as $\Z[U]$-modules. 
\end{theorem}

We recall the definition of $U$-knots from \cite{oz-sz-rational} Definition 5.2.

\begin{definition}
Let $Y$ be a rational homology three-sphere. A $U$-knot $K\subset Y$ is a knot such that for any $\xi\in\underline{\Spin^c}(Y,K)$, $CFK^\infty(Y,K,\xi)$ is relatively $\Z\oplus\Z$-filtered chain homotopy equivalent to a $U$-tower $R$, which is the free, rank one $\Z[U,U^{-1}]$-module with zero differential.
\end{definition}

\begin{proposition}[\cite{oz-sz-rational}, Corollary 5.3]
If $K_2\subset Y_2$ is a $U$-knot, then for each $\xi_1\in\underline{\Spin}^c(Y_1,\overline{K}_1)$ and $\frak s_2\in\Spin^c(Y_2)$, there is some $\xi_2\in\underline{\Spin^c}(Y_2)$ representing $\frak s_2$, with the property that
$$CFK^\infty(Y_1,\overline{K_1},\xi_1)\cong CFK^\infty(Y_1\#Y_2,\overline{K_1}\#\overline{K_2},\xi_1\#\xi_2)$$ as $\Z\oplus\Z$-filtered chain complexes.
\end{proposition}

\begin{lemma}[\cite{oz-sz-rational}, Lemma 7.1]
$O_{q/r}\subset -L(q,r)$ is a $U$-knot. There is an affine identification $\phi$ of the $\Spin^c$ structures on $\widehat{HFK}(L(q,r),O_{q/r})$ with $\Z$ such that
$$\widehat{HFK}(L(q,r),O_{q/r},\phi(i))\cong\begin{cases}
    \Z & \text{if $0\le i\le q-1$,} \\
    0 & \text{otherwise.}
\end{cases}$$
\end{lemma}

Let $K\subset Y$ be a null-homologous knot. As in section 2.2 of [OS10], $H^2(Y,K)$ is defined to be $H^2(Y-\nd(K),\partial\nd(K))$ and is an affine space for $\underline{\Spin^c}(Y,K)$. There is a natural map $\underline{\Spin^c}(Y,K)\to\Spin^c(Y)$ which allows us to identify $$\Spin^c(Y)\cong\frac{\underline{\Spin^c}(Y,K)}{\Z\cdot\PD[\mu]}.$$ Fix a $\Spin^c$ structure $\frak t$ over $Y$, and under the above identification, let $H^2(Y,K,\frak t)\cong\Z$ denote the subgroup of $H^2(Y,K)$ corresponding to $\frak t$, and also $H^2(Y\#-L(q,r),K\# O_{q,r},\frak t)\cong\Z$.

With this notation, the following more general version of \cite[Lemma 7.2]{oz-sz-rational} holds with analogous proof.

\begin{lemma}[\cite{oz-sz-rational}, Lemma 7.2]
Let $K\subset Y$ be a null-homologous knot. Under the connected sum
$$(Y,K)\#(-L(q,r),O_{q,r})\to(Y\#-L(q,r),K\# O_{q,r}),$$ the diagram in Figure~\ref{fig:h2-rational-commutative-diagram} commutes where $f(x,y)=qx+y$. Moreover, under this correspondence, if $K_\lambda$ is the pushoff of $K$ with respect to the integral framing $a$, then $\PD[K_\lambda]$ represents the element $p\in\Z\cong H^2(Y\#-L(q,r),K\# O_{q,r})$, where $p/q=a+r/q$.

\begin{figure}[h]
\centering
\begin{tikzcd}
\mathbb{Z}\oplus\mathbb Z \arrow[rrr, "f"] \arrow[d, "\cong"] &  &  & \mathbb Z \arrow[d, "\cong"]           \\
{H^2(Y,K,\mathfrak t)\oplus H^2(-L(q,r),O_{q,r})} \arrow[rrr] &  &  & {H^2(Y\# -L(q,r),K\# O_{q,r},\frak t)}
\end{tikzcd}
\caption{}
\label{fig:h2-rational-commutative-diagram}
\end{figure}

\end{lemma}

\begin{proof}
Follows as in \cite{oz-sz-rational}.
\end{proof}

With these ingredients, the proof of Theorem~\ref{thm:rational-surgery-formula} follows as in \cite{oz-sz-rational}.

\section{Cosmetic Surgeries}

In this section, we prove our main results. For simplicity, we consider coefficients in the field $\F=\F_2$. Throughout, we let $K\subset Y$ be a null-homologous knot in a closed, oriented three-manifold. We will obstruct the result of Dehn surgeries on $K$ from being homeomorphic to each other using the rank of $\wh{HF}(Y_r(K))$ for various positive rational numbers $r$.

We state some lemmas that will be useful in calculating the homology of the mapping cones for Dehn surgery. 

The following lemma was proven in \cite{ni-property-g} Equations (5) and (6).

\begin{lemma}
If $s\le t$, then $\im(\wh v_s)_*\subset\im(\wh v_t)_*$ and $\im(\wh h_s)_*\supset\im(\wh h_t)_*$.
\end{lemma}

As a consequence of the hat version of Lemma~\ref{lem:v-h-equiv}, $\rank H_*(\wh A_s)=\rank H_*(\wh A_{-s})$ and $\rank (\wh v_s)_*=\rank(\wh h_{-s})_*$ for all $s$.

\begin{lemma}\label{lem:hat-plus-iso}
Suppose for some $s\in\Z$ that $\wh v_s$ or $\wh h_s$ induces an isomorphism on homology. Then $v_s^+$ or $h_s^+$ also induces an isomorphism on homology, respectively.
\end{lemma}

\begin{proof}
We prove this lemma for $\wh v_s$, and the case for $\wh h_s$ is proven similarly. We will show by induction that $(v_{s}^\delta)_*$ is an isomorphism for all $\delta\ge 0$, which by Lemma 2.7 in \cite{oz-sz-integer} implies that $(v_{g-1}^+)_*$ is an isomorphism. Suppose that $(v_{s}^\delta)_*$ is an isomorphism for some $\delta\ge 0$ (the base case $\delta=0$ is $v_{s}^0=\wh v_{s}$). 

Clearly, the maps $\wh v_s$, $v_s^{\delta+1}$, and $v_s^\delta$ fit into the commutative diagram of short exact sequences in Figure~\ref{fig:v-delta-ses}. 
Therefore, the diagram in Figure~\ref{fig:v-hat-not-isomorphism} commutes and the rows are exact, so by the five lemma, $(v_{s}^{\delta+1})_*$ is also an isomorphism. 
\end{proof}

\begin{figure}[h]
\centering
\begin{tikzcd}
0 \arrow[r] & \wh A_s \arrow[r, "i"] \arrow[d, "\wh v_s"] & A_s^{\delta+1} \arrow[r, "U"] \arrow[d, "v_s^{\delta+1}"] & A_s^\delta \arrow[r] \arrow[d, "v_s^\delta"] & 0 \\
0 \arrow[r] & \wh B \arrow[r, "i"] & B^{\delta+1} \arrow[r, "U"] & B^\delta \arrow[r] & 0
\end{tikzcd}
\caption{}
\label{fig:v-delta-ses}
\end{figure}

\begin{figure}[h]
\centering
\begin{tikzcd}
H_*(A_{s}^\delta) \arrow[r, "\partial"] \arrow[d, "(v_{s}^\delta)_*"] & H_*(\widehat A_{s}) \arrow[r, "i_*"] \arrow[d, "(\widehat v_{s})_*"] & H_*(A_{s}^{\delta+1}) \arrow[r, "U_*"] \arrow[d, "(v_{s}^{\delta+1})_*"] & H_*(A_{s}^\delta) \arrow[r, "\partial"] \arrow[d, "(v_{s}^\delta)_*"] & H_*(\widehat A_{s}) \arrow[d, "(\widehat v_{s})_*"] \\
H_*(B^\delta) \arrow[r, "\partial"]                                  & H_*(\widehat B) \arrow[r, "i_*"]                                         & H_*(B^{\delta+1}) \arrow[r, "U_*"]                                  & H_*(B^\delta) \arrow[r, "\partial"]                                  & H_*(\widehat B)                                        
\end{tikzcd}
\caption{}
\label{fig:v-hat-not-isomorphism}
\end{figure}

\begin{lemma}[\cite{ni-property-g} Lemma 4.1]\label{lem:blow-up-unknot-prop}
Let $K\subset Y$ be a null-homologous knot. Suppose that $F$ is a taut Seifert surface for $K$. Let $X=(Y\times[0,1])\# N\overline{\C P^2}$ for some $N$, and let $\iota:(Y,K)\to(X,K\times\{0\})$ be the inclusion map. Let $G\subset X$ be a properly-embedded smooth connected surface with $\partial G=K\times\{0\}$, $[G]=\iota_*[F]\in H^2(X,K\times\{0\})$, and $g(G)<g(F)$. Then $$\im(v_k^+)_*\supset \im(h_k^+)_*,$$ when $k\ge g(G)$. Moreover, if $G$ is a disk, then
$$(v_0^+)_*=(h_0^+)_*$$
\end{lemma}

The same is true for the hat version of knot Floer homology by essentially the same proof.

Again, call a knot where such $(X,G)$ for $G$ a disk exists a blow-up bounding knot. We also consider the unknot itself to be a blow-up bounding knot.

The following argument is due to Ni:

\begin{lemma}\label{lem:rank-a-b}
For $K\subset Y$ a blow-up unknotted knot, $(\wh v_0)_*:H_*(\wh A_0)\to H_*(\wh B)$ is surjective.
\end{lemma}
\begin{proof}
Denote by $O$ the unknot in $Y$. There is a sequence of $N$ crossing changes changing a positive crossing to a negative one that changes $O$ to the knot $K$. Let $W_n'(K):Y_n(K)\to Y$ and $W_n'(O):Y_n(O)\to Y$ denote the natural two-handle cobordisms. Fix a $\Spin^c$ structure $\frak t$ on $Y$, let $\frak x_0(O)\in\Spin^c(W_n'(O))$ and $\frak x_0(K)\in\Spin^c(W_n'(K))$ as in Section 2 restrict to the $\Spin^c$ strctures $\frak x_0(O)\vert Y=\frak x_0(K)\vert Y=\frak t$ and $\frak t_0(O)=\frak x_0(O)\vert Y_n(O)$, $\frak t_0(K)=\frak x_0(K)\vert Y_n(K)$. Let $\frak S$ be the $\Spin^c$ structure on $W_n'(O)\# N\overline{\C P^2}$ given by the connected sum of $\frak x_0(O)$ with $N$ copies of $\frak s_1\in\Spin^c(\overline{\C P^2})$, where $\frak s_1$ satisfies that $c_1(\frak s_1)$ generates $H^2(\overline{\C P^2})$. 

By Kirby calculus, the Kirby diagram of the $n$-framed unknot in $Y$ blown up $N$ times is equivalent to the diagram of the $n$-framed knot $K$ in $Y$. Thus, the cobordism $W_n'(O)\# N\overline{\C P^2}$ is the composition of two cobordisms
$$Y_n(O)\xrightarrow{X_n(K)}Y_n(K)\xrightarrow{W_n'(K)}Y$$ for some cobordism $X_n(K)$. The $\Spin^c$ structure $\frak S$ restricts to $\frak t_0(O)$ on $Y_n(O)$, $\frak t_0(K)$ on $Y_n(K)$, and $\frak t$ on $Y$. Therefore, $X_n(K)$ induces a cobordism map 
$$\Phi=\wh F_{X_n(K),\frak S\vert X_n(K)}:\wh{HF}(Y_n(O),\frak t_0(O))\to\wh{HF}(Y_n(K),\frak t_0(K))$$ that by Theorem~\ref{thm:oz-sz-2.3} fits into the commutative diagram in Figure~\ref{fig:blow-up-unknotted}. By a sample calculation, $(\wh v_0(O))_*$ is surjective. Therefore $(\wh v_0(K))_*$ must also be surjective. Direct summing over all $\frak t\in\Spin^c(Y)$ proves the statement.

\begin{figure}[h]
\centering
\begin{tikzcd}
{\wh{HF}(Y_n(O),\frak t_0(O))} \arrow[rr, "\Phi"] \arrow[d, "\cong"] &  & {\wh{HF}(Y_n(K),\frak t_0(K))} \arrow[d, "\cong"] \\
H_*(\wh A_{0,\frak t}(O)) \arrow[rr] \arrow[d, "(\wh v_0(O))_*"]      &  & H_*(\wh A_{0,\frak t}(K)) \arrow[d, "(\wh v_0(K))_*"] \\
H_*(\wh B_{\frak t}(O)) \arrow[rr, "="]                           &  & H_*(\wh B_{\frak t}(K))                          
\end{tikzcd}
\caption{}
\label{fig:blow-up-unknotted}
\end{figure}

\end{proof}

The following result and its proof are analogous to \cite{oz-sz-rational} Proposition 9.7.
\begin{proposition}\label{prop:genus}
Let $K\subset Y$ be a null-homologous knot. Then
$$g(K)=\max(0,\{s\in\Z\ \vert\ (\wh v_{s-1})_*\text{ is not an isomorphism}\}).$$ 
\end{proposition}
In particular, $K$ is the unknot if and only if $(\wh v_s)_*$ is an isomorphism for all $s\ge 0$.

\begin{proof}
Let $g=g(K)$, by Lemma~\ref{lem:genus-finiteness}, 
$$g(K)\ge\max(0,\{s\in\Z\ \vert\ (\wh v_{s-1})_*\text{ is not an isomorphism}\}).$$ 

If $g(K)=0$, i.e. $K$ is the unknot, then $(\wh v_s)_*$ is an isomorphism for all $s\ge 0$. Otherwise, if $g(K)>0$, it suffices to show $(\wh v_{g-1})_*$ is not an isomorphism. Consider the short exact sequence
\begin{align}\label{ses:v-g-1}
0\to C\{i<0,j\ge g-1\}\to C\{i\ge 0 \text{ or } j\ge g-1\}\xrightarrow{v_{g-1}^+}C\{i\ge 0\}\to 0.
\end{align}
The middle term is $A_{g-1}^+$, and $C\{i\ge 0\}=B^+$. The assumption that $C$ is reduced together with the $U$-isomorphism and the fact that knot Floer homology detects genus implies that $C\{i,j\}$ is nonzero only if $-g\le i-j\le g$. Therefore,
$$H_*(C\{i<0, j\ge g-1\})\cong H_*(C\{(-1,g-1)\}),$$ which by the $U$-isomorphism is isomorphic to $$H_*(C\{0,g\})=\wh{HFK}(K,g),$$ and $\wh{HFK}(K,g)\ne 0$ since knot Floer homology detects genus. Therefore, $(v_{g-1}^+)_*$ is not an isomorphism. By Lemma~\ref{lem:hat-plus-iso}, $(\wh v_{g-1})_*$ is also not an isomorphism.

\end{proof}

We abuse notation to let $\wh A_{j/q}$ denote $\wh A_{\lfloor j/q\rfloor}$, and also $\wh B_{j/q}$, $\wh v_{j/q}$, $\wh h_{j/q}$ similarly.

\begin{lemma}\label{lem:cancellation-general}
Let $K\subset Y$ be a blow-up bounding knot, $p,q>0$ relatively prime, and let $x\in H_*(\wh A_{j/q})$ for some $j$. If $\lfloor j/q\rfloor\ge 0$, there exists $\alpha\in H_*(\oplus_{i\ge j+p}\wh A_{i/q})$ such that $(\wh D_{p/q})_*(x+\alpha)=(\wh v_{j/q})_*(x)$ (where we think of $H_*(\wh B_{j/q})$ as including into $H_*(\wh\B)$). Similarly, if $\lfloor j/q\rfloor\le 0$, there exists $\alpha\in H_*(\oplus_{i\le j-p}\wh A_{i/q})$ such that $(\wh D_{p/q})_*(x+\alpha)=(\wh h_{j/q})_*(x)$.
\end{lemma}
\begin{proof}
Since $K$ is blow-up bounding, by Lemma~\ref{lem:blow-up-unknot-prop}, for $s\ge 0$, we have $\im(\wh h_s)_*\subset\im(\wh v_s)_*\subset\im(\wh v_{s+1})_*\subset\dots$, and for $s\le 0$, $\im(\wh v_s)_*\subset\im(\wh h_s)_*\subset\im(\wh h_{s-1})_*\subset\dots$. Thus, in the first case, there exists $y_1\in H_*(\wh A_{(j+p)/q})$ such that $(\wh h_{j/q})_*(x)=(\wh v_{(j+p)/q})_*(y_1)$. If $(\wh h_{(j+p)/q})_*(y_1)=0$, then stop and set $\alpha=y_1$, otherwise keep going. Inductively, given $y_n$ in $H_*(\wh A_{(j+np)/q})$, there exists $y_{n+1}\in H_*(\wh A_{(j+(n+1)p)/q})$ such that $(\wh h_{(j+np)/q})_*(y_n)=(\wh v_{(j+(n+1)p)/q})_*$. By Lemma~\ref{lem:genus-finiteness}, this process eventually terminates and gives an element $\alpha=\sum_{n\ge 1}y_n\in H_*(\oplus_{i\ge j+p}\wh A_{i/q})$ with the desired condition.

Similarly, in the case where $\lfloor j/q\rfloor\le 0$, since $\im(\wh v_{j/q})_*\subset\im(\wh h_{(j-p)/q})_*$, we can inductively add cancellation terms to the left. 
\end{proof}

Using the mapping cone formula and the above lemmas, in the case where $K$ is blow-up bounding, we may now calculate the rank of $\wh{HF}(Y_r(K))$ for any positive rational number $r$.

\begin{lemma}
Let $K\subset Y$ be a blow-up bounding knot, let $p,q>0$ coprime. Let $b=\rank(H_*(\wh B))=\rank\wh{HF}(Y)$. Then
\begin{align}\label{eq:rank-formula-pos-general}
\begin{split}
    \rank \wh{HF}(Y_{p/q}(K)) &= q(\rank\ker(\wh v_0)_*+b-\rank(\wh v_0)_*) \\
    &+2q\sum_{s=1}^{g-1}(\rank\ker(\wh v_s)_*+b-\rank(\wh v_s)_*)+ 2t_K^{p/q}-pb,
\end{split}
\end{align}
where
\begin{align}
    t_K^{p/q} &= \sum_{j=0}^{p-1} \rank\left(\im(\wh v_{j/q})_* \cap \im(\wh h_{(j-p)/q})_*\right).
\end{align}
\end{lemma}

Note that if $Y=S^3$, Equation~\ref{eq:rank-formula-pos-general} agrees with Equation 40 in \cite{oz-sz-rational}. Indeed, define $\nu(K)$ as in \cite{oz-sz-rational}. If $\nu(K)>0$, then
$$\rank\left(\im(\wh v_{j/q})_* \cap \im(\wh h_{(j-p)/q})_*\right) = \begin{cases}
    1 & j\ge q\nu(K) \text{ and } j<p-q(\nu(K)-1) \\
    0 & \text{otherwise.}
\end{cases}$$ Thus,
$$t_K^{p/q}=\max(0,p-(2\nu(K)-1)q).$$ Also, for all $s\ge 0$,
$$\rank\ker(\wh v_s)_*=\begin{cases}
    \rank H_*(\wh A_s)_* & s<\nu(K) \\
    \rank H_*(\wh A_s)_*-1 & s\ge\nu(K),
\end{cases}$$ so the two equations agree.

If $\nu(K)=0$, then $t_K^{p/q}=p$ and $\rank\ker(\wh v_s)_*=\rank H_*(\wh A_s)_*-1$ for all $s\ge 0$, so again the two equations agree.

\begin{proof}
By Theorem~\ref{thm:rational-surgery-formula}, $\wh{HF}(Y_{p/q}(K))\cong H_*(\wh{\mathbb X}_{p/q})$, the homology of the mapping cone of $\wh D_{p/q}$ from $\wh{\A}=\oplus_{i\in\Z/p\Z}\wh\A_i$ to $\wh{\B}=\oplus_{i\in\Z/p\Z}\wh\B_i$ (we have omitted the slope $p/q$ from the notation and summed over all $\Spin^c$ structures on $Y$). So
$$\wh\A=\bigoplus_{i=0}^{p-1}\prod_{r\in\Z}\wh A_{(i+pr)/q}=\prod_{j\in\Z}\wh A_{j/q}$$ has $q$ copies of $\wh A_s$ for each $s\in\Z$, and under the map $\wh v$, each $\wh A_{j/q}$ corresponds to a copy of $\wh B$ in $\wh\B$, which we shall denote $\wh B_{j/q}$. Then $\wh h_{j/q}$ maps $\wh A_{j/q}$ to $\wh B_{(j+p)/q}$. Thus, for each $s\in\Z$, there are $p$ copies of $\wh B_s$ in the codomain of $\wh h_{s-1}$ and $q-p$ copies of $\wh B_s$ in the codomain of $\wh h_s$.

The finiteness of $H_*(\wh\X_{p/q})$ is guaranteed by Lemma~\ref{lem:genus-finiteness}; in fact, it implies that $\wh\X_{p/q}$ is isomorphic to a truncated version of the mapping cone
$$\wh\X_{p/q}^c=MC(\wh D_{p/q}^c:\wh\A_{p/q}^c\to\wh\B_{p/q}^c),$$ for some large $c$, where
$$\wh\A_{p/q}^c=\bigoplus_{j=-qc+1}^{qc-1}\wh A_{j/q}=\bigoplus_{s=-c+1}^{c-1}\bigoplus_{i=0}^{q-1}\wh A_s, \hspace{5mm} \wh\B_{p/q}^c=\bigoplus_{j=-qc+p+1}^{qc-1}\wh B_{j/q},$$
and $\wh D_{p/q}^c$ is the restriction of $\wh D_{p/q}$ to $\wh A_{p/q}^c$, with the exception that for any $x\in\wh A_{j/q}$ with $-qc+1\le j\le-qc+p$, $\wh D_{p/q}^c(x)=\wh h_{j/q}(x)$, and for any $x\in\wh A_{j/q}$ with $qc-p\le j\le qc-1$, $\wh D_{p/q}^c(x)=\wh v_{j/q}(x)$. Note that it suffices to take $c\ge g+p/q+1$, and
$$\rank H_*(\wh\A_{p/q}^c) = q\sum_{s=-c+1}^{c-1}\rank H_*(\wh A_s), \hspace{5mm} \rank H_*(\wh\B_{p/q}^c) = qb(2c-1)-pb.$$

We construct a basis for $\ker(\wh D_{p/q})_*$ of the entire mapping cone, which has the same rank as $\ker(\wh D_{p/q}^c)_*$ for large $c$ by Lemma~\ref{lem:genus-finiteness}. For $j\ge 0$, let
$$\{ x_{j/q}^i\}_{i=1}^{\rank\ker(\wh v_{j/q})_*}$$
be a basis of $\ker(\wh v_{j/q})_*$. For each such $x_{j/q}^i$, choose an $\alpha_{j/q}^i\in H_*(\oplus_{k\ge j+p}\wh A_{k/q})$ according to Lemma~\ref{lem:cancellation-general} such that $\wt x_{j/q}^i=x_{j/q}^i+\alpha_{j/q}^i$ is in $\ker(\wh D_{p/q})_*$. For $j<0$, let 
$$\{x_{j/q}^i\}_{i=1}^{\rank\ker(\wh h_{j/q})_*}$$ 
be a basis of $\ker(\wh h_{j/q})_*$. For each such $x_{j/q}^i$, choose a $\beta_{j/q}^i\in H_*(\oplus_{k\le j-p}\wh A_{k/q})$ according to Lemma~\ref{lem:cancellation-general} such that $\wt x_{j/q}^i=x_{j/q}^i+\beta_{j/q}^i$ is in $\ker(\wh D_{p/q})_*$. Lastly, for $0\le j\le p-1$, let 
$$\{y_{j/q}^i\}_{i=1}^{\rank\left(\im(\wh v_{j/q})_* \cap \im(\wh h_{(j-p)/q})_*\right)}$$
be a basis of a subspace of $H_*(\wh A_{j/q})$ that is mapped isomorphically to $\im(\wh v_{j/q})_* \cap \im(\wh h_{(j-p)/q})_*$ under $(\wh v_{j/q})_*$. Choose $z_{j/q}^i\in H_*(\wh A_{(j-p)/q})$ such that $(\wh v_{j/q})_*(y_{j/q}^i)=(\wh h_{(j-p)/q})_*(z_{j/q}^i)$. Choose $\gamma_{j/q}^i\in H_*(\oplus_{k\ge j+p}\wh A_{k/q})$ and $\delta_{j/q}^i\in H_*(\oplus_{k\le j-p}\wh A_{k/q})$ according to Lemma~\ref{lem:cancellation-general} such that $\wt y_{j/q}^i=y_{j/q}^i+z_{j/q}^i+\gamma_{j/q}^i+\delta_{j/q}^i$ is in $\ker(\wh D_{p/q})_*$.

The set $$S=\{\wt x_{j/q}^i\}\cup\{\wt y_{j/q}^i\}$$ is linearly independent, and it is finite by Lemma~\ref{lem:genus-finiteness}. Moreover, it spans $\ker(\wh D_{p/q})_*$: let $z\in\ker(\wh D_{p/q})_*$ be nonzero, and suppose first that $z\in H_*(\oplus_{k\ge 0}\wh A_{k/q})$. Let $z_{j/q}$ denote the term of $z$ in $H_*(\wh A_{j/q})$, and let
$$j=\min\{k\ge 0\ \vert\ z_{k/q}\ne 0\}.$$ Then $z\in\ker(\wh D_{p/q})_*\cap H_*(\oplus_{k\ge j}\wh A_{k/q})$, so $z_{j/q}\in\Span \{x_{j/q}^i\}_{i=1}^{\rank\ker(\wh v_{j/q})_*}$: otherwise if $z_{j/q}\notin\ker(\wh v_{j/q})_*$, then since $z\in H_*(\oplus_{k\ge j}\wh A_{k/q})$, we have $z_{(j-p)/q}=0$ so $$(\wh v_{j/q})_*(z_{j/q})+(\wh h_{(j-p)/q})_*(z_{(j-p)/q})\ne 0\in H_*(\wh B_{j/q})$$ which means $z\notin\ker(\wh D_{p/q})_*$. Thus, we can subtract some combination of the basis elements $\wt x_{j/q}^i$ to eliminate the $H_*(\wh A_{j/q})$ component of $z$:
$$z-\sum_{i\in J}\wt x_{j/q}^i\in\ker(\wh D_{p/q})_*\cap H_*(\oplus_{k>j}\wh A_{k/q})$$ for some indexing set $J\subset\{1,\dots,\rank\ker(\wh v_{j/q})_*\}$. Inductively, we can then write 
$$z=\sum_{k\in I}\sum_{i_k\in J_k}\wt x_{k/q}^{i_k}$$ for some finite indexing sets $I\subset\{k\ge j\}$, $J_k\subset\{1,\dots,\rank\ker(\wh v_{k/q})_*\}$. Similarly, if $z\in H_*(\oplus_{k<0}\wh A_{k/q})$, then we can write
$$z=\sum_{k\in I}\sum_{i_k\in J_k}\wt x_{k/q}^{i_k}$$ for some finite indexing set $I\subset\{k\le j\}$, $J_k\subset\{1,\dots,\rank\ker(\wh h_{k/q})_*\}$.

Now for the general case $z\in\ker(\wh D_{p/q})_*$, we claim that for each $0\le j\le p-1$, $$z_{j/q}\in\Span\left(\{y_{j/q}^i\}\cup\{x_{j/q}^i\}\right).$$ Indeed, suppose this is not true for some $0\le j\le p-1$, then $(\wh v_{j/q})_*(z_{j/q})\notin\im(\wh h_{(j-p)/q})_*$, so $(\wh D_{p/q})_*(z)\ne 0$, a contradiction. Thus,
\begin{align*}
z-\sum_{j=0}^{p-1}\left(\sum_{i\in J_j}\wt y_{j/q}^i+\sum_{i\in J_j'}\wt x_{j/q}^i\right) &\in \left(\ker(\wh D_{p/q})_*\cap H_*(\oplus_{k\ge p}\wh A_{k/q})\right) \\
&\hspace*{5mm}+\left(\ker(\wh D_{p/q})_*\cap H_*(\oplus_{k<0}\wh A_{k/q})\right).
\end{align*} By the above argument, $z$ is in the span of $S$.

Therefore, $S$ is a basis for $\ker(\wh D_{p/q})_*$, so since $\rank\ker(\wh v_s)_*=\rank\ker(\wh h_{-s})_*$ by the hat version of Lemma~\ref{lem:v-h-equiv} and these are $0$ when $s\ge g$ by Lemma~\ref{lem:genus-finiteness},
$$\rank\ker(\wh D_{p/q})_*=q\rank\ker(\wh v_0)_*+2q\sum_{s=1}^{g-1}\rank\ker(\wh v_s)_*+t_K^{p/q}.$$

Since $\rank(\wh v_s)_*=\rank H_*(\wh A_s)-\rank\ker(\wh v_s)_*$ and $\rank\coker(\wh D_{p/q}^c)_*=\rank H_*(\wh\B_{p/q}^c)-\rank H_*(\wh\A_{p/q}^c)+\rank\ker(\wh D_{p/q}^c)_*$ for some large $c$, it follows that
\begin{align*}
    \rank\coker(\wh D_{p/q}^c)_* &= b(2c-1)q-pb-q\sum_{s=-c+1}^{c-1} \rank H_*(\wh A_s) \\
        &\hspace*{5mm}+ q\rank\ker(\wh v_0)_*+2q\sum_{s=1}^{g-1}\rank\ker(\wh v_s)_*+t_K^{p/q} \\
        &= q(b-\rank H_*(\wh A_0)+\rank\ker(\wh v_0)_*) \\
        &\hspace*{5mm}+ 2q\sum_{s=1}^{g-1}(b-\rank H_*(\wh A_s)+\rank\ker(\wh v_s)_*) +t_K^{p/q}-pb \\
        &= q(b-\rank(\wh v_0)_*)+2q\sum_{s=1}^{g-1}(b-\rank(\wh v_s)_*)+t_K^{p/q}-pb,
\end{align*}
where the second equality uses that $\rank H_*(\wh A_s)=b$ for all $|s|\ge g$ by Lemmas~\ref{lem:v-h-equiv} and ~\ref{lem:genus-finiteness}. Therefore, because 
$$\rank(H_*(\wh\X_{p/q}))=\rank\ker(\wh D_{p/q})_*+\rank\coker(\wh D_{p/q})_*$$ and $$\rank\ker(\wh D_{p/q})_*=\rank\ker(\wh D_{p/q}^c)_*,\hspace*{5mm}\rank\coker(\wh D_{p/q})_*=\rank\coker(\wh D_{p/q}^c)_*$$ for large enough $c$, equation~\ref{eq:rank-formula-pos-general} holds.

\end{proof}

Now we are ready to prove our main theorems. The main ingredient for each result is Equation~\ref{eq:rank-formula-pos-general}.

\begin{theorem}
Let $K\subset Y$ be a blow-up bounding knot, and suppose $0<r<s\le 1$ are distinct positive rational numbers with the property that $Y_r(K)\cong\pm Y_s(K)$. Then $K$ is the unknot. Furthermore, if $K\subset Y$ is a blow-up unknotted knot and $0<r<s$ are distinct positive rational numbers such that $Y_r(K)\cong\pm Y_s(K)$, then $r,s>1$ or $K$ is the unknot.
\end{theorem}
\begin{proof}
Let $K\subset Y$ be blow-up bounding. Since $H_1(Y_{\pm p/q}(K))\cong H_1(Y)\oplus\Z/p\Z,$ we can fix $p>0$ throughout. First, let $q$ be relatively prime to $p$ and $q\ge p$. Then $\im(\wh v_{j/q})_* \cap \im(\wh h_{(j-p)/q})_*=\im(\wh v_0)_*$ for all $j=0,\dots,p-1$. Thus $t_K^{p/q}=p\rank(\wh v_0)_*$. Therefore, according to equation~\ref{eq:rank-formula-pos-general}, for positive integral $q$ with $0<p/q\le 1$, $\rank\wh{HF}(Y_{p/q}(K))$ is a nondecreasing function of $q$. The function is strictly increasing unless $(\wh v_s)_*$ is an isomorphism for all $s\ge 0$. Hence, by Proposition~\ref{prop:genus}, $K$ is the unknot.

Since the total rank of $\wh{HF}$ is an invariant of the underlying (unoriented) three-manifold, the first part of the theorem is true. Also, the second part is true for $0<r<s\le 1$ since blow-up unknotted knots are blow-up bounding.

Now suppose $K$ is blow-up unknotted and $q,q'$ are relatively prime to $p$ and $0<q<p<q'$. Then, with $b=\rank(H_*(\wh B))$,
\begin{align}\label{eq:t-ineq}
\begin{split}
    t_K^{p/q}-t_K^{p/q'} &\le \sum_{j=0}^{p-1} \rank(\wh v_{j/q})_* - p\rank(\wh v_0)_* \\
    &= \sum_{j=q}^{p-1} \rank(\wh v_{j/q})_* - (p-q)\rank(\wh v_0)_* \\
    &\le (p-q)(b-\rank(\wh v_0)_*) \\
    &= 0
\end{split}
\end{align} since $(\wh v_0)_*$ is surjective by Lemma~\ref{lem:rank-a-b}. By equation~\ref{eq:rank-formula-pos-general},
\begin{align*}
    \rank\wh{HF}(Y_{p/q'}(K))&-\rank\wh{HF}(Y_{p/q}(K)) = (q'-q)(\rank\ker(\wh v_0)_*+b-\rank(\wh v_0)_*) \\
    &+ (q'-q)\left(2\sum_{s=1}^{g-1}(\rank\ker(\wh v_s)_*+b-\rank(\wh v_s)_*)\right) \\
    &+ 2t_K^{p/q'}-2t_K^{p/q}. \\
\end{align*}
The entire sum from $s=1$ to $g-1$ on the right hand side is nonnegative, so we ignore this term, and by Lemma~\ref{lem:rank-a-b}, we have 
$$\rank\ker(\wh v_0)_*+b-\rank(\wh v_0)_*= \rank\ker(\wh v_0)_*.$$ Together with inequality~\ref{eq:t-ineq}, we get
\begin{align*}
    \rank\wh{HF}(Y_{p/q'}(K))-\rank\wh{HF}(Y_{p/q}(K)) &\ge (q'-q)\rank\ker(\wh v_0)_*\\
    &\ge 0,
\end{align*} with equality only if $(\wh v_s)_*$ is an isomorphism for all $s\ge 0$, hence $K$ is the unknot by Lemma~\ref{prop:genus}.

Since the total rank of $\wh{HF}$ is an invariant of the underlying (unoriented) three-manifold, the result holds.
\end{proof}

In the case where $p$ is $1$ or $2$, the following corollary is immediate, since there do not exist distinct positive integers $q,q'$ coprime to $p$ with $p/q$ and $p/q'$ both greater than $1$.

\begin{corollary}
Let $K\subset Y$ be a blow-up unknotted knot, and suppose $0<q'<q$ are distinct positive numbers with the property that $Y_{p/q}(K)\cong\pm Y_{p/q'}(K)$ for $p\in\{1,2\}$, then $K$ is the unknot.
\end{corollary}

\begin{theorem}
Let $K\subset Y$ be a nontrivial blow-up bounding knot and $q\ge 2$, then $Y_{1/q}(K)\ncong\pm Y$. Furthermore, if $K$ is a nontrivial blow-up unknotted knot, then $Y_1(K)\ncong\pm Y$.
\end{theorem}
\begin{proof}
Let $q\ge 2$, then by Equation~\ref{eq:rank-formula-pos-general},
$$\rank\wh{HF}(Y_{1/q}(K))\ge (b-\rank(\wh v_0)_*)(q-1)+\rank(\wh v_0)_*+q\rank\ker(\wh v_0)_*\ge b,$$ and it is strictly greater unless $(\wh v_s)_*$ is an isomorphism for all $s\ge 0$. Hence, by Proposition~\ref{prop:genus}, $K$ is the unknot.

If $q=1$, then
$$\rank\wh{HF}(Y_1(K))\ge \rank(\wh v_0)_*+\rank\ker(\wh v_0)_*=\rank H_*(\wh A_0)\ge b$$ by Lemma~\ref{lem:rank-a-b} since $K$ is blow-up unknotted, and it is strictly greater unless $(\wh v_s)_*$ is an isomorphism for all $s\ge 0$, which implies $K$ is the unknot.

Since the total rank of $\wh{HF}$ is an invariant of the underlying (unoriented) three-manifold, the result holds.
\end{proof}

\textbf{Acknowledgements} The author would like to thank Yi Ni for advising him in this project and the referee for many helpful comments.

\bibliographystyle{alpha} 
\bibliography{biblio}

\end{document}